\documentclass{article}
\usepackage{amsthm, amssymb}
\usepackage{amsmath}
\usepackage{amsfonts}
\usepackage{graphicx}

\newcommand{\R}{\mathbb{R}}
\newcommand{\Rn}{{\mathbb{R}^n}}
\newcommand{\B}{\mathbb{B}}
\newcommand{\Bn}{{\mathbb{B}^n}}

\newcommand{\Hn}{{\mathbb{H}^n}}
\newcommand{\Hnb}{{\mathbb{H}^n_b}}
\newcommand{\arcsinh}{\textrm{arc\,sinh\,}}
\newcommand{\arccosh}{\textrm{arc\,cosh\,}}

\newtheorem{theorem}{Theorem}[section]
\newtheorem{lemma}[theorem]{Lemma}         
\newtheorem{corollary}[theorem]{Corollary} 
\newtheorem{remark}[theorem]{Remark}       

\newtheorem{conjecture}[theorem]{Conjecture}
\newtheorem{proposition}[theorem]{Proposition}

\theoremstyle{remark}
\newtheorem{example}[theorem]{Example} 

\title{Hyperbolic type distances in starlike domains}
\author{Riku Kl\'en}
\date{}

\begin{document}

\maketitle           

\renewcommand{\theequation}{\thetheorem} 
                    
\makeatletter
\let \c@equation=\c@theorem

\begin{abstract} 
  We study the growth of hyperbolic type distances in starlike domains. We derive estimates for various hyperbolic type distances and consider the asymptotic sharpness of the estimates.
\end{abstract}



\section{Introduction}

The hyperbolic distance has turned out to be a useful tool in geometric function theory. The basic models for the hyperbolic distance are the unit ball model and the upper half space model. Using these models in the plane case $n=2$, we can find the hyperbolic distance in any domain with at least 2 boundary points via the Riemann mapping theorem. In higher dimensions $n \ge 3$, there are no such results we could use to consider the hyperbolic distance in general domains. A solution to this is to use other distance functions, which approximate the hyperbolic distance and are easier to evaluate. We call this kind of distance functions hyperbolic type distances.

The study of hyperbolic distances was initiated four decades ago by Gehring, Palka, Martin and Osgood \cite{GehPal76,GehOsg79,MarOsg1986}. Thereafter many researchers have studied hyperbolic type metrics or used them as a tool in their work, see for example \cite{AndVamVuo1988,BeaMR1488447,KosNie2005,Vai2007}.

In this article we are interested in the growth of hyperbolic type distances in proper subdomains $G \subsetneq \Rn$. We consider the growth along a Euclidean line segment from $z' \in G$ to $z \in \partial G$. By a linear transformation we may assume that $z'=0$. To ensure that the line segment $[z',z)$ is in $G$ we restrict our study to starlike domains, which means that for every $y \in G$ the Euclidean line segment $[x,y]$ is contained in $G$.

Let $G \subset \Rn$ be a starlike domain and $z \in \partial G$. Let $y \in [0,z)$ and denote $t = |y|$. We study the behaviour of the function
\begin{equation}\label{function f(t)}
  f_m(t) = m(0,y) = m(0,tz/|z|),
\end{equation}
\begin{figure}[ht!]
  \begin{center}
    \includegraphics[width=.5\textwidth]{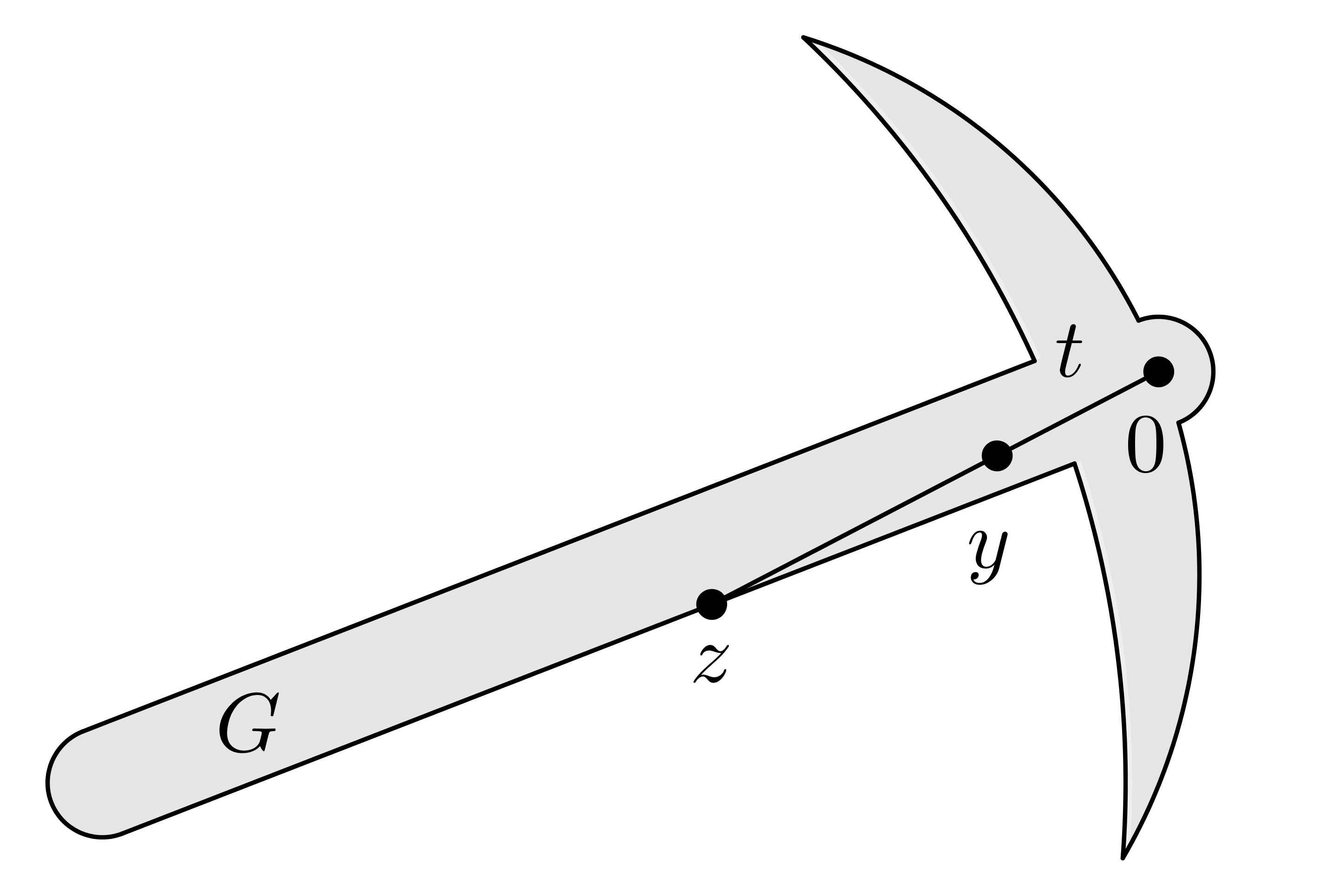}
    \caption{Starlike domain $G$ and the function $f_m(t) = m(0,tz/|z|) = m(0,y)$.}\label{fig:function f}
  \end{center}
\end{figure}

\noindent
where $m$ is a hyperbolic type distance, see Figure \ref{fig:function f}. Note that now $f_m$ is a continuous mapping from $[0,|z|)$ to $[0,\infty)$. To study this function we start of with an easier one
\begin{equation}\label{function g(t)}
  g(t) = d_G(y) = d(y,\partial G) = d(tz/|z|,\partial G).
\end{equation}

All the results of our study are true in a more general setting as long as the line segment $[0,z)$ is contained in the domain. However, we consider starlike domains as then this condition is clearly fulfilled.

Our main result is the following Schwarz lemma type theorem. For definition of the distances see the section named after the distance.

\begin{theorem}\label{thm:main theorem}
  Let $G \subsetneq \Rn$ be a domain with $0 \in G$ and $f_m(t)$ be the function defined in \eqref{function f(t)} for any $z \in \partial G$. For $m \in \{ j,k,\widetilde{\sigma}, c \}$
  \[
    f_m'(0) = \frac{1}{d_G(0)}
  \]
  and for $m \in \{ \alpha, \delta \}$
  \[
    \frac{1}{d_G(0)} \le f_m'(0) \le \frac{2}{d_G(0)},
  \]
  where the upper and lower bounds are best possible, and
  \[
    0 \le f_{\widetilde{\tau}}'(0) \le \frac{1}{d_G(0)}.
  \]
\end{theorem}

\section{Preliminaries and the hyperbolic distance, $\rho_G$}

In this section we introduce notation and consider examples of the hyperbolic distance in the upper half space and a ball.

For $a,b \in \Rn$ we denote the closed Euclidean line segment between the points by $[a,b] = \{ c \in \Rn \colon c=x+t(y-x), \, 0 \le t \le 1 \}$. We also use notation $(a,b)$, $[a,b)$ and $(a,b]$ for open and half-closed line segments in $\Rn$. We denote the smallest angle between line segments $[a,b]$ and $[b,c]$ by $\measuredangle(a,b,c)$.

We denote Euclidean balls and spheres with centre $x \in \Rn$ and radius $r > 0$, respectively, by $B^n(x,r)$ and $S^{n-1}(x,r)$. A domain $G \subsetneq \Rn$, $0 \in G$, is said to be starlike, if it is strictly starlike with respect to $0$: for every $z \in \partial G$ the line segment $[0,z)$ is contained in $G$. We say that distance function $m_G$ in $G \subsetneq \in \Rn$ is hyperbolic type, if $m_\Bn \approx \rho_\Bn$, where $\rho_\Bn$ is the hyperbolic distance in the unit ball defined in \eqref{eqn:rho in Bn}.

Let $b < 0$. The hyperbolic distance for all $u,w \in \Hnb = \{ a \in \Rn \colon a_n>b \}$ is defined as
\[
  f_{\rho_\Hnb}(t) = \rho_\Hnb (u,w) = \arccosh \left( 1+\frac{|u-w|^2}{2 |u_n-b| |w_n-b|} \right).
\]
Note that we use shifted version of the upper half space, as the usual upper half space $\Hn = \{ a \in \Rn \colon a_n>0 \}$ does not contain the origin.

\begin{example}\label{exa:rho in Hnb}
Let $z = (0,0,\dots,0,b) \in \partial G$ and $y \in [0,z)$ with $|y|=t$. We show that in this case $f_{\rho_\Hnb}'(0) = 1/d_\Hnb(0)$.

Now for $t \in (0,d_\Hnb(0))$ we have
\begin{eqnarray*}
  f_{\rho_\Hnb}(t) = \rho_\Hnb (0,y) & = & \arccosh \left( 1+\frac{|y|^2}{2 |b| |y_n-b|} \right)\\
  & = & \arccosh \left( 1+\frac{t^2}{2 d_\Hnb(0) (d_\Hnb(0)-t)} \right)
\end{eqnarray*}
and by differentiation we obtain
\[
  f_{\rho_\Hnb}'(t) = \frac{1}{d_\Hnb(0)-t}.
\]
We consider $f_{\rho_\Hnb}'(0)$ by taking the limit $f_{\rho_\Hnb}'(t) \to 1/d_\Hnb(0)$ as $t \to 0$. \hfill $\triangle$
\end{example}
Let $B=B^n(x,r)$ be a ball with $x \in \Rn$ and $r >0$. For $u,w \in B$ the hyperbolic distance is defined by
\begin{equation}\label{eqn:rho in Bn}
  \rho_B(u,w) = \arcsinh \frac{|u-w|/r}{\sqrt{1-|u|^2/r^2}\sqrt{1-|w|^2/r^2}}.
\end{equation}

\begin{example}\label{exa:rho in B}
Let $z \in \partial B$ and $x \in [0,z)$. We show that in this case $f_{\rho_B}'(0) = 1/d_B(0)$.

Now for $t \in (0,d_B(0))$
\begin{eqnarray*}
  f_{\rho_B}(t) = \rho_B (0,y) & = & \arcsinh \frac{|y|/r}{\sqrt{1-|y|^2/r^2}}\\
  & = & \arcsinh \frac{t/d_B(0)}{\sqrt{1-t^2/d_B(0)^2}}
\end{eqnarray*}
and $f_{\rho_B}'(t) = d_B(0)/(d_B(0)^2-t^2)$. Taking the limit we obtain $f_{\rho_B}'(t) \to 1/d_B(0)$ as $t \to 0$.  \hfill $\triangle$
\end{example}

Examples \ref{exa:rho in Hnb} and \ref{exa:rho in B} suggest that for hyperbolic type distances $f_m'(0) = a/d_G(0)$ could hold for $a \ge 0$ or perhaps even for $a=1$. It turns out that this conjecture is not true in general. However, based on our study in this article, it seems that the following is true
\begin{conjecture}
  For a hyperbolic type distance $m$ there exists constants $a \ge 0$ and $b \ge 1$ such that
\[
  \frac{a}{d_G(0)} \le f_m'(0) \le \frac{b}{d_G(0)}.
\]
\end{conjecture}

\section{Distance to the boundary function, $d_G$}\label{sec:distance to the boundary}

In this section we study the problem for the distance to the boundary function $g(t)$ formulated in \eqref{function g(t)}. Let us begin our study with few simple planar examples, which are easy to reconstruct in higher dimensions.

\begin{example}\label{exm:simple domain}
We consider starlike domain
\[
  G = \R^2 \setminus R_x, \quad R_x = \{ w = (w_1,w_2) \in R^2 \colon w_1 \ge x_1, \, w_2 \ge x_2 \},
\]
where $x = (x_1,x_2) \in \R^2$ with $x_1,x_2 > 0$. For a given $z = z(\alpha)$ we derive a formula for $g(t)$ and show that $g(t) = (|z|-t) \sin \alpha$ and thus linear for large values of $t$. We choose $z = (x_1,y_2)$ for $z_2 \ge x_2$ and note that $z \in \partial G$.

\begin{figure}[ht!]
  \begin{center}
    \includegraphics{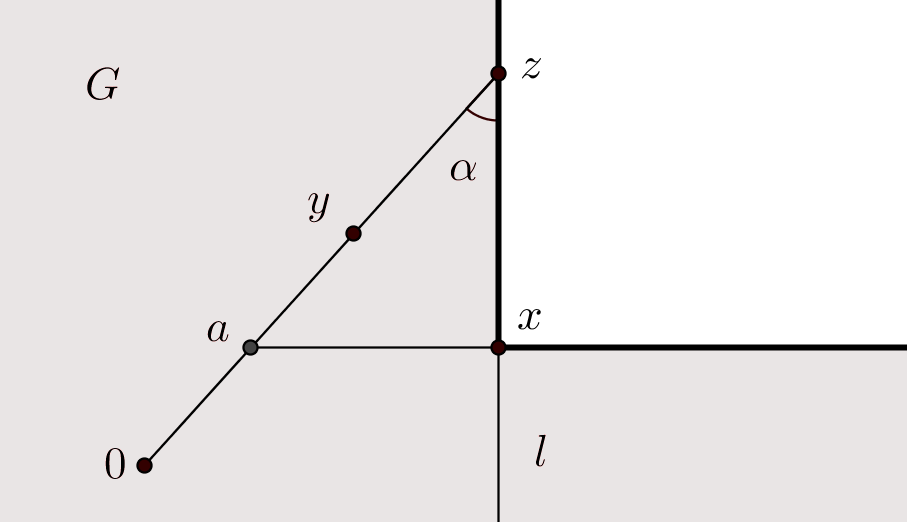}\vspace{3mm}
    \includegraphics[scale=.19]{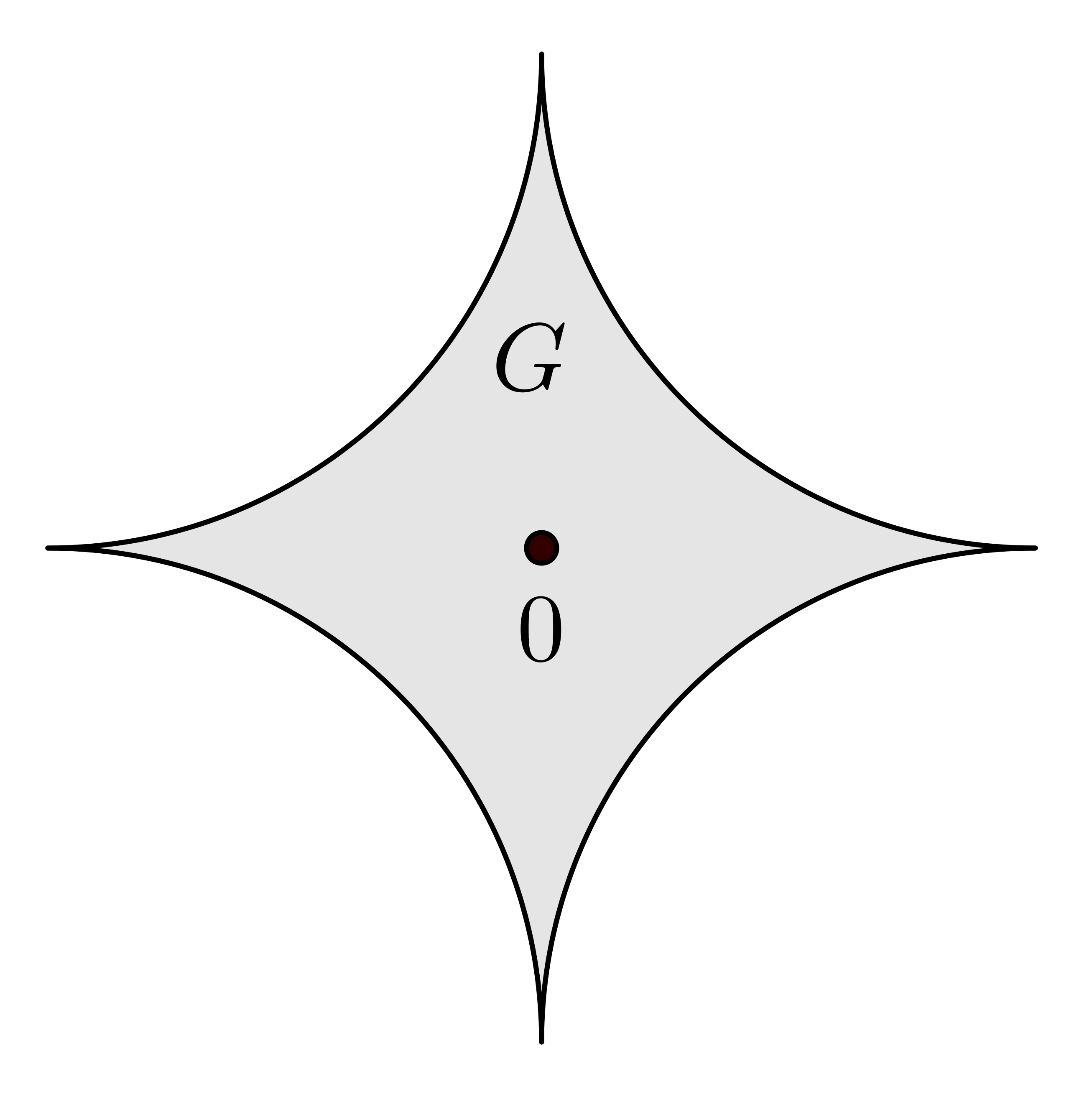}
    \caption{The domains $G$ of Examples \ref{exm:simple domain} (on left) and \ref{exm:circular domain} (on right).}\label{fig:simple example}
  \end{center}
\end{figure}

When $t$ is small, then $y$ is close to $0$ and $g(t) = |y-x|$. To be more explicit we can write
\[
  g(t) = d_G(y) = |y-x| = \sqrt{|a-x|^2 + |a-y|^2}, \quad \textrm{if }t \le |a|,
\]
where $a$ is the point in $[0,z]$, with $a_2 = x_2$, see Figure \ref{fig:simple example}.

If $t$ is faraway from $0$, the we have
\[
  g(t) = d(y,l) = |y-z| \sin \alpha,\quad \textrm{if } t \ge |a|,
\]
where $l = \{ (x_1,t) \colon t \in \R \}$ and $\alpha = \measuredangle (y,z,x)$.

Next we want to express our function $g(t)$ in terms of $t$ and points $x$ and $z$. We easily obtain $|y-z| = |z|-t$, $|a-z| = |x-z| \cos \alpha$, $|a-x| = |x-z| \sin \alpha$, $|a-y| = |z|-t-|a-z|$ and
\[
  |a| = |z|-|a-z| = |z| - |x-z| \cos \alpha.
\]
Putting all together gives
\[
  g(t) = \left\{ \begin{array}{ll}
    \sqrt{t^2 -2t(|z|-|a-z|) + |z|(|z|-2 |a-z|) + |x-z|^2}, & \textrm{if } t \le |a|,\\
    (|z|-t) \sin \alpha, & \textrm{if } t > |a|,
  \end{array} \right.
\]
and here $|a-z| = |x-z| \cos \alpha$. \hfill $\triangle$
\end{example}

\begin{example}\label{exm:circular domain}
  Let us consider starlike domain $G=\B^2 \setminus \overline{(B^2(1+i,1)} \cup \overline{B^2(-1+i,1)} \cup \overline{B^2(-1-i,1)} \cup \overline{B^2(1-i,1))}$ and $z = 1 \in \partial G$, see Figure \ref{fig:simple example}.
  
  For $y \in [0,z)$ and $t =|y| \in [0,1)$ we have
  \[
    g(t) = d_G(y) = |y-(1+i)|-1 = \sqrt{t^2-2t+2}-1.
  \]
  Now
  \[
    g'(t) = \frac{t-1}{\sqrt{t^2-2t+2}}
  \]
  implying $g'(t) \to 0$ as $t \to 1$. \hfill $\triangle$
\end{example}

\begin{example}\label{exm:polynomial domain}
  In Example \ref{exm:circular domain} we could place the circular arc with any decreasing function: First define a decreasing function $h \colon [0,1] \to [0,1]$ with $h(0)=1$ and $h(1)=0$. Then reflect the function across the real axis and then reflect both function across the imaginary axis.
  
  For $p \in \{ 1,2,\dots \}$ we can define polynomial function $h(t) = (1-t)^p$ to obtain a domain $G$. Now for $z = 1 \in \partial G$ we have
  \[
    g(t) = \sqrt{(1-t)^{2p}+(1-t)^{4p-2}}.
  \]
  and since $\sqrt{(1-t)^{2p}+(1-t)^{4p-2}} \ge \sqrt{(1-t)^{2p}} = (1-t)^p$ we obtain
  \begin{eqnarray*}
    g'(t) & = & \frac{-2p(1-t)^{2p-1}-(4p-2)(1-t)^{4p-3}}{2\sqrt{(1-t)^{2p}+(1-t)^{4p-2}}}\\
    & \le & \frac{-2p(1-t)^{2p-1}-(4p-2)(1-t)^{4p-3}}{2 (1-t)^p}\\
    & = & -p(1-t)^{p-1}-(2p-1)(1-t)^{3p-3} \to 0
  \end{eqnarray*}
  as $t \to 1$. \hfill $\triangle$
\end{example}

Let us then consider general case $G \subsetneq \Rn$. How quickly and how slowly $g(t)$ can decrease for large values of $t$? How quickly and how slowly $g(t)$ can increase for $t$ close to $0$? Before considering the bounds for $g(t)$, we introduce angular domain.

For $x,y \in \Rn$ and $\alpha \in (0,\pi)$ we define
\[
  S_{\alpha,x,y} = \{ z \in \Rn \colon \measuredangle(y,x,z) \le \alpha \}.
\]

If $t$ is close to $0$, then the slowest growth for $g(t)$ occurs in the following case:
\begin{equation}\label{eqn:domain G1}
  G_1=\Rn \setminus (S_{\pi/4,-x,-2x} \cup S_{\pi/4,3x,4x})
\end{equation}
for some $x \in \Rn \setminus \{ 0 \}$ and $z=3x$, see Figure \ref{fig:domain G2}. Now for $t \in [0,|x|]$, we have
\begin{equation}\label{eqn:upper bound for d(0)}
  g(t) = d_G(0)+t = |x|+t.
\end{equation}

\begin{figure}[ht!]
  \begin{center}
    \includegraphics{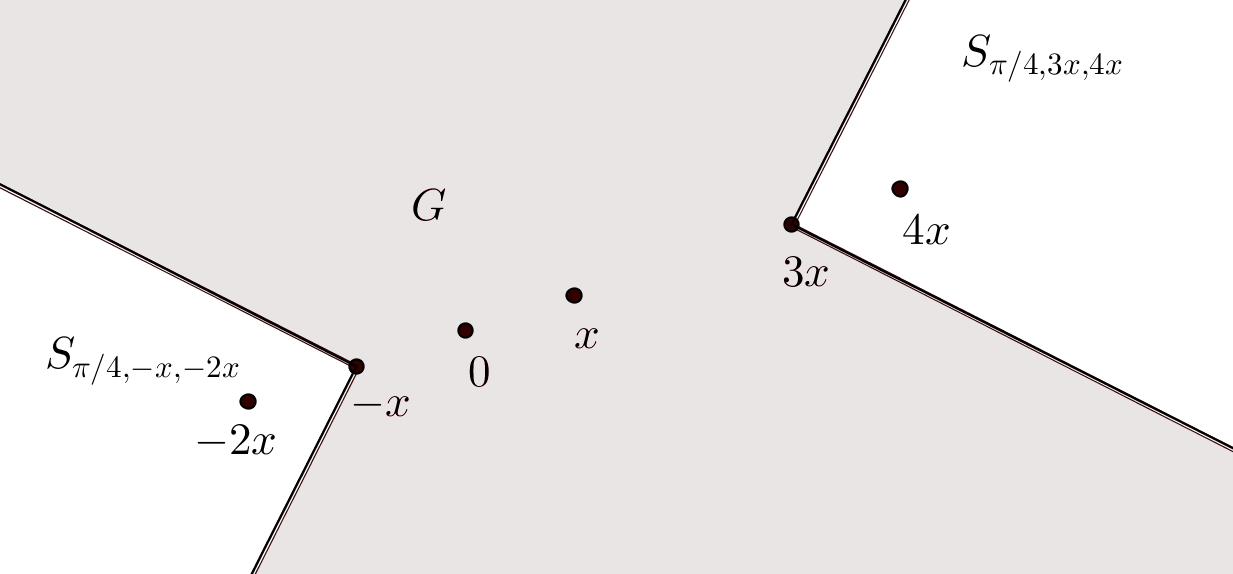}
    \caption{The domain $G_1=\Rn \setminus (S_{\pi/4,-x,-2x} \cup S_{\pi/4,3x,4x})$ defined in \eqref{eqn:domain G1} for the extremal case, when $t$ is close to $0$.}\label{fig:domain G2}
  \end{center}
\end{figure}

If $t$ is faraway from the boundary, then the slowest growth occurs in the case
\begin{equation}\label{eqn:domain G2}
  G_2=\Rn \setminus S_{\pi/4,x,2x}
\end{equation}
for any $x \in \Rn \setminus \{ 0 \}$ and $z=x$. Now
\begin{equation*}
  g(t) = d_G(0)-t = |x|-t
\end{equation*}
for $t \in [0,|z|)$.

As we will later see, it turns out that the domains $G_1$ and $G_2$ defined respectively in \eqref{eqn:domain G1} and \eqref{eqn:domain G2} can be used as extremal domains for many hyperbolic type distances. Note that if $G$ was not starlike, we could use $G_1' = \Rn \setminus \{ -x,3x \}$ and $G_2' = \Rn \setminus \{ x \}$ instead of $G_1$ and $G_2$.

Next we consider how quickly $g(t)$ can decrease. Close to the origin the fastest decrement occurs in the domain $G=B^n(0,d_G(0))$ and it is
\begin{equation}\label{eqn:upper bound for d(|x|)}
  g(t) = d_G(0)-t
\end{equation}
for $t \in [0,d_G(0))$.
 
Faraway from the origin the decrement can made arbitrarily slow as can be observed from Examples \ref{exm:simple domain}, \ref{exm:circular domain} and \ref{exm:polynomial domain}. We have arrived to the following theorem:

\begin{theorem}\label{thm:distance to the boundary estimates}
  For any starlike domain $G$ and point $z \in \partial G$ we have
  \[
    g'(t) \le 1, \quad \textrm{for } t \in[0,|z|).
  \]
  Moreover,
  \[
    -1 \le g'(0) \le 1 \quad \textrm{and} \quad -1 \le \lim_{t \to |z|} g'(t) \le 0.
  \]
\end{theorem}

\section{Distance ratio distance, $j_G$}

In this section we estimate $f_j(t)$ for the distance ratio distance $j_G$, which is defined in any open subset $G \subsetneq \Rn$ for points $u,w \in G$ by
\[
  j_G(u,w) = \log \left( 1+\frac{|u-w|}{\min \{ d_G(u),d_G(w) \}} \right).
\]
The distance ratio distance was introduced by Vuorinen in the 1980's \cite{Vuo1985} and in a slightly different form by Gehring and Osgood \cite{GehPal76}.

For $f_j(t)$ we can use the same domains as for $g(t)$ to consider the extremal cases. Note that the following result is true also in non starlike domains.

\begin{theorem}\label{thm:distance ratio estimates}
  Let $G \subsetneq \Rn$ be a starlike domain and $z \in \partial G$. For the distance ratio distance $j_G$ we have $f_j'(0) = 1/d_G(0)$,
  \[
    f_j(t) \ge \log \left( 1+ \frac{t}{d_G(0)} \right) \quad \textrm{for all }t \in [0,|z|),
  \]
  and
  \[
    f_j(t) \le \log \left( \frac{d_G(0)}{d_G(0)-t)} \right) \quad \textrm{for all }t \in [0,d_G(0)).
  \]
\end{theorem}
\begin{proof}
  Let us first consider lower bound for $f_j(0)$. We denote a closest boundary $x \in \partial G$ with $d_G(0) = |0-x|$ and $z \in \partial G$. Now $y \in [0,|z|)$ with $|y|=t$ we have
  \begin{eqnarray*}
    f_j(t) & = & j_G(0,y) = \log \left( 1+\frac{|y|}{\min \{ d_G(0),d_G(y)  \}} \right)\\
    & \ge & \log \left( 1+\frac{t}{d_G(0)} \right) := l(t).
  \end{eqnarray*}
  This is obtained for example as in \eqref{eqn:upper bound for d(0)} in the domain $G_1=\Rn \setminus (S_{\pi/4,-x,-2x} \cup S_{\pi/4,3x,4x})$ for some $x \in \Rn \setminus \{ 0 \}$ and $z=3x$.
  
  Let us then consider upper bound for $f_j(0)$. Now as $d_G(0)$ is a constant we may assume $d_G(y) \le d_G(0)$ and obtain by \eqref{eqn:upper bound for d(|x|)}
  \begin{eqnarray*}
    f_j(t) & = & \log \left( 1+\frac{|y|}{\min \{ d_G(0),d_G(y)  \}} \right) \le \log \left( 1+\frac{t}{d_G(y)} \right)\\
    & = & \log \left( 1+\frac{t}{d_G(0)-t} \right) =\log \left( \frac{d_G(0)}{d_G(0)-t)} \right) =: u(t).
  \end{eqnarray*}
  This situation is obtained in the domain $G'=B^n(0,d_G(0))$, because for any $G$ we have $G' \subset G$.
  
  The upper and lower bounds of $f_j(t)$ give us
  \[
    \lim_{t \to 0} l'(t) = \lim_{t \to 0} \frac{1}{d_G(0)+t} \le f_j'(0) \le \lim_{t \to 0} u'(t) = \lim_{t \to 0} \frac{1}{d_G(0)-t}
  \]
  which yields $f_j'(0) = 1/d_G(0)$.
\end{proof}

Note that the lower bound for $f_j(t)$ in Theorem \ref{thm:distance ratio estimates} works for all $t$ whereas the upper bound holds only close to the origin. As a matter of fact, by Theorem \ref{thm:distance to the boundary estimates}, the function $g'(t)$ can be arbitrarily close to zero near the boundary and therefore we cannot find an upper bound for $f_j(t)$ in this case.

\section{Quasihyperbolic distance, $k_G$}

Next we consider the quasihyperbolic distance. Let $G \subsetneq \Rn$ be a domain. We define the quasihyperbolic length of a curve $\gamma \subset G$ by
\[
  \ell_k(\gamma) = \int_\gamma \frac{|du|}{d_G(u)}.
\]
For $u,w \in G$ the quasihyperbolic distance between $u$ and $w$ is define by
\[
  k_G(u,w) = \inf_{\gamma_{uw}} \ell(\gamma_{uw}),
\]
where the infimum is taken over all rectifiable curves $\gamma_{uw}$ joining $u$ and $w$ in $G$. The quasihyperbolic distance was introduced in the 1970's by Gehring and Palka \cite{GehPal76}.

\begin{theorem}\label{thm:quasihyperbolic estimates}
  Let $G \subsetneq \Rn$ be a starlike domain and $z \in \partial G$. For the quasihyperbolic distance $k_G$ we have $f_k'(0) = 1/d_G(0)$,
  \[
    f_k(t) \ge \log \left( 1+ \frac{t}{d_G(0)} \right) \quad \textrm{for all }t \in [0,|z|),
  \]
  and
  \[
    f_k(t) \le \log \left( \frac{d_G(0)}{d_G(0)-t)} \right) \quad \textrm{for all }t \in [0,d_G(0)).
  \]
\end{theorem}
\begin{proof}
  We start by finding a lower bound for $f_k(t)$. Let $z \in \partial G$ and $y \in [0,z)$. For $t=|z|$ we estimate  
  \begin{eqnarray*}
    f(t) & = & k_G(0,y) = \int_\gamma \frac{|du|}{d_G(u)} \ge \int_{[0,y]} \frac{|du|}{d_G(0)+u} = \int_0^t \frac{du}{d_G(0)+u}\\
    & = & \log(d_G(0)+t)-\log d_G(0) = \log \left( 1+\frac{t}{d_G(0)} \right).
  \end{eqnarray*}
  For $G' = B^n(0,d_G(0))$ and $y \in [0,d_G(0))$ clearly $d_{G'}(0) = d_G(0)$ and we obtain
  \begin{eqnarray*}
    f(t) & = & k_G(0,tz/|z|) \le k_{G'}(0,tz/|z|) = \int_0^t \frac{ds}{d_G(0)-s}\\
    & = & -\log(d_G(0)-t)+\log d_G(0) = \log \frac{d_G(0)}{d_G(0)-t}.
  \end{eqnarray*}
  The lower and the upper bounds of $f_k(t)$ are equal to the corresponding bounds for $f_j(t)$ and the expression for $f_k'(0)$ follows as in the proof of Theorem \ref{thm:distance ratio estimates}.
\end{proof}

Exactly same observation for $f_k(t)$ near the boundary can be made that we made for the distance ratio distance. Theorem \ref{thm:quasihyperbolic estimates} gives a lower bound and an upper bound can not be obtained.

\section{Triangular ratio distance, $s_G$ ($\sigma_G$)}

For a domain $G \subsetneq \Rn$ and points $u,w \in G$ we define the triangular ratio distance by
\[
  s_G(u,w) = \sup_{q \in \partial G} \frac{|u-w|}{|u-q|+|q-w|}.
\]
Geometrically the supremum is attained at a point $q$ such that it is either on line segment $[u,w]$ or if this is not possible, then $q$ is on the largest ellipsoid with focii $u$ and $w$ contained in $G$. The triangular ratio distance was introduced by H\"ast\"o in the 2000's \cite{Hasto2002}. Since $s_G \in [0,1]$ we observe that $s_G$ is not hyperbolic type. For $u,w \in G$ we define
\[
  \sigma_G(u,w) = \tan \frac{\pi s_G(u,w)}{2}
\]
to obtain a hyperbolic type distance.

As in the case of the distance to the boundary function and the distance ratio distance we can use the same extremal domains. The domain $G_1=\Rn \setminus (S_{\pi/4,-x,-2x} \cup S_{\pi/4,3x,4x})$ for some $x \in \Rn \setminus \{ 0 \}$ and $z=3x$ gives lower bound
\begin{equation}\label{s lower bound}
  f_s(t) \ge \frac{t}{d_G(0)+(d_G(0)+t)} = \frac{t}{2 d_G(0)+t}
\end{equation}
for $t \in [0,|z|)$.

The domain $G_2=\Rn \setminus S_{\pi/4,x,2x}$ for any $x \in \Rn \setminus \{ 0 \}$ and $z=x$ gives upper bound
\begin{equation}\label{s upper bound}
  f_s(t) \le \frac{t}{d_G(y)+(d_G(y)-t)} = \frac{t}{2d_G(0)-t}
\end{equation}
for $t \in [0,d(0))$.

Combining these estimates we obtain:

\begin{theorem}\label{thm:sigma distance}
  Let $G \subsetneq \Rn$ be a starlike domain and $z \in \partial G$. Then
  \[
    f_\sigma(t) \ge \tan \frac{\pi t}{4 d_G(0)+2t}
  \]
  for $t \in [0,|z|)$ and
  \[
    f_\sigma(t) \le \tan \frac{\pi t}{4d_G(0)-2t}
  \]
for $t \in [0,d(0))$. Moreover, $f_\sigma'(0) = \pi/(4d_G(0))$.
\end{theorem}
\begin{proof}
  The bounds for $f_\sigma(t)$ follow from \eqref{s lower bound} and \eqref{s upper bound}. By differentiation we obtain
  \[
    \frac{\partial}{\partial t} \tan \frac{\pi t}{4 d_G(0) \pm 2t} = \frac{\pi d_G(0) + 4 \pi d_G(0) \tan^2 \displaystyle \frac{\pi t}{4 d_G(0) \pm 2t}}{4 d_G(0)^2 \pm 4 d_G(0) t+t^2}
  \]
  and thus $f_\sigma´(0) = \pi/(4d_G(0))$.
\end{proof}

\begin{remark}\label{rem:sigma remark}
  Theorem \ref{thm:sigma distance} suggests that an alternative way to define a hyperbolic type distance by using the triangular ratio distance, could be $\widetilde{\sigma}_G(u,w) = \tfrac{4}{\pi} \sigma_G(u,w)$. For this distance function $f_{\widetilde{\sigma}}'(0) = 1/d_G(0)$.
\end{remark}

\section{Cassinian distance, $c_G$}

For a domain $G \subsetneq \Rn$ and points $u,w \in G$ we define the Cassinian distance by
\[
  c_G(u,w) = \sup_{q \in \partial G} \frac{|u-w|}{|u-q||q-w|}.
\]
The Cassinian distance was introduced by Ibragimov in the 2000's \cite{Ibragimov2009}. In the plane case the supremum is attained at a point $q$ that is on the largest Cassinian oval with focii $u$ and $w$ contained in $G$.

\begin{theorem}\label{thm:Cassinian distance}
  Let $G \subsetneq \Rn$ be a starlike domain and $z \in \partial G$. Then
  \[
    f_c(t) \ge \frac{t}{d_G(0)(d_G(0)+t)}
  \]
  for $t \in [0,|z|)$ and
  \[
    f_c(t) \le \frac{t}{d_G(0)(d_G(0)-t)}
  \]
  for $t \in [0,d(0))$. Moreover, $f_c'(0) = 1/d_G(0)$.
\end{theorem}
\begin{proof}
  The domain $G_1=\Rn \setminus (S_{\pi/4,-x,-2x} \cup S_{\pi/4,3x,4x})$ for some $x \in \Rn \setminus \{ 0 \}$ and $z=3x$ gives lower bound
  \[
    f_c(t) \ge \frac{t}{d_G(0)(d_G(0)+t)}
  \]
  for $t \in [0,|z|)$.

  The domain $G_2=\Rn \setminus S_{\pi/4,x,2x}$ for any $x \in \Rn \setminus \{ 0 \}$ and $z=x$ gives upper bound
  \[
    f_c(t) \le \frac{t}{d_G(y)(d_G(y)+t)} = \frac{t}{d_G(0)(d_G(0)-t)}
  \]
  for $t \in [0,d(0))$.
  
  Differentiation gives
  \[
    \frac{\partial}{\partial t} \frac{t}{d_G(0)(d_G(0) \pm t)} = \frac{d_G(0)}{d_G(0)^2 \pm 2 d_G(0) t + t^2}
  \]
  and clearly $f_c'(0) = 1/d_G(0)$.
\end{proof}

\section{Apollonian distance, $\alpha_G$}

Let $G \subset \overline{\Rn}$ be a domain such that $\partial G$ is not contained in a sphere in $\overline{\Rn}$. Then for $u,w \in G$ we define the Apollonian distance \cite[Theorem 1.1]{BeaMR1488447} by
\[
  \alpha_G(u,w) = \sup_{a,b \in \partial G} \log \frac{|a-w||b-u|}{|a-u||b-w|}.
\]
The Apollonian distance was introduced by Barbilian in the 1930's \cite{Barbilian1935} and reintroduced by Beardon in the 1990's in connection with the hyperbolic metric \cite{BeaMR1488447}. It is worth pointing out that we can write
\begin{equation}\label{eqn:apollonian formula}
  \alpha_G(u,w) = \sup_{a \in \partial G} \log \frac{|a-w|}{|a-u|} + \sup_{b \in \partial G} \log \frac{|b-u|}{|b-w|}
\end{equation}
and in the case $n=2$ each quotient defines an Apollonian circle. Taking the supremum means that we take the largest possible Apollonian circles in $\overline{G}$.

\begin{lemma}\label{lem:apollonian estimate}
  Let $u \in \Rn$, $r > 0$ and $w \in B^n(u,r)$. Denote $S=S^{n-1}(u,r)$. Then
  \[
    \sup_{a \in S} \frac{|w-a|}{|u-a|} = \frac{r+|u-w|}{r}, \quad \sup_{a \in S} \frac{|u-a|}{|w-a|} = \frac{r}{r-|u-w|},
  \]
  and for $r' > r$ and $S'=S^{n-1}(u,r')$
  \[
    \sup_{a \in S'} \frac{|w-a|}{|u-a|} \le \sup_{a \in S} \frac{|w-a|}{|u-a|}, \quad \sup_{a \in S'} \frac{|u-a|}{|w-a|} \le \sup_{a \in S} \frac{|u-a|}{|w-a|}.
  \]
\end{lemma}
\begin{proof}
  To simplify notation we may assume that $u=0$ and $w = \beta e_1$ for $\beta \in (0,r)$.
  
  Now
  \begin{equation}\label{apollonian estimate1}
    \sup_{a \in S} \frac{|w-a|}{|u-a|} = \max_{\zeta \in [0,\pi]} \left\{ \frac{|w-re^{i \zeta}|}{r} \right\} = \frac{|w-re^{i \pi}|}{r} = \frac{r+|w|}{r},
  \end{equation}
  because $|w|=|u-w|<r$ and the function $f(\zeta) = |w-re^{i \zeta}|$ is increasing for $\zeta \in [0,\pi]$. Similarly we obtain
  \begin{equation}\label{apollonian estimate2}
    \sup_{a \in S} \frac{|u-a|}{|w-a|} = \max_{\zeta \in [0,\pi]} \left\{ \frac{r}{|w-re^{i \zeta}|} \right\} = \frac{r}{|w-re^{i 0}|} = \frac{r}{r-|w|}.
  \end{equation}
  
  Next we note that the function $g(r) = (r+\gamma)/r$, $\gamma > 0$, is decreasing on $(0,\infty)$, because $g'(r) = -\gamma/(r \gamma+r^2) <0$. The function $h(r) = r/(r-\gamma)$, $\gamma > 0$, is decreasing on $(\gamma,\infty)$, because $h'(r) = \gamma/(\gamma r-r^2) < 0$.
  
  Combining monotonicity of $g$ with \eqref{apollonian estimate1} gives
  \[
    \sup_{a \in S'} \frac{|w-a|}{|u-a|} = \frac{r'+|u-w|}{r'} \le \frac{r+|u-w|}{r} = \sup_{a \in S} \frac{|w-a|}{|u-a|}
  \]
  and similarly monotonicity of $h$ with \eqref{apollonian estimate2} gives
  \[
    \sup_{a \in S'} \frac{|u-a|}{|w-a|} = \frac{r'}{r'-|w|} \le \frac{r}{r-|w|} = \sup_{a \in S} \frac{|u-a|}{|w-a|}. \qedhere
  \]
\end{proof}

\begin{theorem}\label{thm:f for alpha}
  Let $G  \subset \Rn$ be a starlike domain and $z \in \partial G$. Then
  \[
    f_\alpha(t) \ge \log \left( 1+\frac{t}{d_G(0)} \right)
  \]
  for $t \in [0,|z|)$ and
  \[
    f_\alpha(t) \le \log \frac{d_G(0)+t}{d_G(0)-t}
  \]
  for $t \in [0,d_G(0))$. Moreover, we have $1 / d_G(0) \le f_\alpha'(0) \le 2 / d_G(0)$ and the upper bound for $f_\alpha(t)$ and the lower bound for $f_\alpha'(0)$ are best possible.
\end{theorem}
\begin{proof}
  By \eqref{eqn:apollonian formula} we can estimate $\alpha_G(0,y)$ by estimating $\sup \log(|a-y|/|a|)$ and $\sup \log(|a|/|a-y|)$ separately.
  
  By Lemma \ref{lem:apollonian estimate} we obtain upper bounds
  \[
    \sup_{a \in \partial G} \frac{|y-a|}{|a|} \le \sup_{a \in S^{n-1}(0,d_G(0))} \frac{|y-a|}{|a|} = \frac{d_G(0)+t}{d_G(0)}
  \]
  and
  \[
    \sup_{a \in \partial G} \frac{|a|}{|y-a|} \le \sup_{a \in S^{n-1}(0,d_G(0))} \frac{|a|}{|y-a|} = \frac{d_G(0)}{d_G(0)-t}.
  \]
  These inequalities give
  \begin{equation}\label{eqn:upper bound fo Apo}
    f_\alpha(t) \le \log \frac{d_G(0)+t}{d_G(0)-t}.
  \end{equation}
  
  We can also use Lemma \ref{lem:apollonian estimate} for lower bound
  \[
    \sup_{a \in \partial G} \frac{|y-a|}{|a|} \ge \inf_{a \in S^{n-1}(0,d_G(0)+t)} \frac{|y-a|}{|a|} = \sup_{a \in S^{n-1}(0,d_G(0)+t)} \frac{|a|}{|y-a|} = \frac{d_G(0)+t}{d_G(0)}
  \]
  and we can trivially estimate
  \[
    \sup_{a \in \partial G} \frac{|a|}{|y-a|} \ge 1.
  \]
  Together these two inequalities give us
  \begin{equation}\label{eqn:lower bound fo Apo}
    f_\alpha(t) \ge \log \frac{d_G(0)+t}{d_G(0)}.
  \end{equation}
  
  Differentiation of \eqref{eqn:upper bound fo Apo} gives $2 d_G(0) / (d_G(0)^2-t^2)$ and differentiation of \eqref{eqn:lower bound fo Apo} gives $1/(d_G(0)+t)$. Thus $1 / d_G(0) \le f_\alpha'(0) \le 2 / d_G(0)$.
  
  Finally, we give two example domains, which show that the upper bound for $f_\alpha(t)$ and the lower bound for $f_\alpha'(0)$ are best possible.
  
  The domain $G' = \Rn \setminus (S_{\pi/4,-x,-2x} \cup S_{\pi/4,x,2x})$ for some $x \in \Rn \setminus \{ 0 \}$ and $z=x$ shows that the upper bound is best possible. Now the suprema in the definition of the Apollonian distance are obtained at $-x$ and $x$, and thus
  \[
    \alpha_{G'}(0,y) = \log \frac{|-x-y||x|}{|-x||x-y|} = \log \frac{d_G(0)+t}{d_G(0)-t}.
  \]
  
  The domain $G_2$ defined in \eqref{eqn:domain G2} shows that the lower bound for $f_\alpha'(0)$ is best possible. We recall that
  \[
    G_2=\Rn \setminus S_{\pi/4,x,2x}
  \]
  for any $x \in \Rn \setminus \{ 0 \}$ and $z=x$. For $u \in \partial G_2$ with $|u-x| = a > 0$ we have
  \[
    \alpha_{G_2}(0,y) = \lim_{a \to \infty} \log \frac{|u-y||x|}{|u||x-y|} = \log \frac{|x|}{|x-y|} = \log \frac{d_G(0)}{d_G(0)-t}
  \]
  and differentiation together with taking the limit gives $1/(d_G(0)-t) \to 1/d_G(0)$ as $t \to 0$.
\end{proof}

Note that in the Theorem \ref{thm:f for alpha} also the upper bound for $f_\alpha'(0)$ is best possible, because the upper bound for $f_\alpha(t)$ is best possible and both upper and lower bounds of $f_\alpha(t)$ tend to $0$ as $t \to 0$.

\section{Seittenranta distance, $\delta_G$}

Let $G \subsetneq \Rn$ be a domain. For $u,w \in G$ we define the Seittenranta distance by
\[
  \delta_G (u,w) = \sup_{a,b \in \partial G} \log \left( 1+ \frac{|a-b||u-w|}{|a-u||b-w|} \right).
\]
The Seittenranta distance was introduced by Seittenranta in the 1990's \cite[Theorem 3.3]{SeiMR1656825}. Before estimating $f_\delta(t)$ we find general upper and lower bound for $\delta_G(x,y)$.

\begin{lemma}\label{lem:upper bound for delta}
  Let $G \subsetneq \Rn$ be a domain. Then for all $u,w \in G$ we have
  \[
    \delta_G(u,w) \le \log \left( 1+\frac{|u-w|}{d_G(u)}+\frac{|u-w|}{d_G(w)}+\frac{|u-w|^2}{d_G(u)d_G(w)} \right).
  \]
\end{lemma}
\begin{proof}
  By the Euclidean triangle inequality we obtain
  \begin{eqnarray*}
    \delta_G(u,w) & \le & \log \left( 1+|u-w|\frac{|a-u|+|u-w|+|w-b|}{|a-u||b-w|} \right)\\
    & = & \log \left( 1+\frac{|u-w|}{|a-u|}+\frac{|u-w|}{|b-w|}+\frac{|u-w|^2}{|a-u||b-w|} \right)\\
    & \le & \log \left( 1+\frac{|u-w|}{d_G(u)}+\frac{|u-w|}{d_G(w)}+\frac{|u-w|^2}{d_G(u)d_G(w)} \right)
  \end{eqnarray*}
  and the assertion follows.
\end{proof}

The following lower bound for $\delta_G(x,y)$ is from \cite[Theorem 3.11]{SeiMR1656825}.

\begin{proposition}\label{pro:lower bound for delta}
  Let $G \subset \Rn$ be a domain and $\partial G$ is not contained in a sphere in $\overline{\Rn}$. Then for all $x,y \in G$ we have
  \[
    \alpha_G(x,y) \le \delta_G(x,y).
  \]
\end{proposition}

We also need exact formulas for the Seittenranta distance in two starlike domains.

\begin{lemma}\label{lem:formula for delta in special domains}
  Let $G_2$ be the domain defined in \eqref{eqn:domain G2} for some $x \in \Rn \setminus \{ 0 \}$ and define a starlike domain
  \[
    G_3 = G_2 \setminus \{ u \in \Rn \colon |u+2x|\le|u| \}.
  \]
  
  \noindent (1) For $y \in [0,x)$ we have
  \[
    \delta_{G_2}(0,y) = \log \left( 1+ \frac{|y|}{|x|} \right).
  \]
  
  \noindent (2) For $y \in [0,x)$ we have
  \[
    \delta_{G_3}(0,y) = \log \left( 1+ \frac{2|x||y|}{|x|(|x|+|y|)} \right).
  \]
\end{lemma}
\begin{proof}
  In both cases we have
  \[
    \delta_{G_i}(0,y) = \sup_{a,b \in \partial G_i} \log \left( 1+\frac{|a-b||y|}{|a||b-y|} \right),
  \]
  where $i \in \{ 2,3 \}$. The idea of our proof is to first show that for any $b \in \partial G_i$ the supremum over $a \in \partial G_i$ is attained at $a=a_0$. Using this property we can find supremum over $b$.

  (1) We denote $r_a = |x-a|$ and $r_b = |x-b|$. For $r_a,r_b > 0$ the angle $\measuredangle(0,x,a)=\measuredangle(0,x,b) = 3\pi/4$ and by the law of cosines
  \[
    |a|^2 = |x|^2+r_a^2+\frac{2|x|r_a}{\sqrt{2}} \quad \textrm{and} \quad |b|^2 = |x|^2+r_b^2+\frac{2|x|r_b}{\sqrt{2}}.
  \]
  We fix $b$ and find the point $a$, $|a|=t$, which gives the minimum value for
  \[
    f(t) = \frac{|a|}{|a-b|}, \quad t > |x|.
  \]
  Since $r_a^2 = t^2-|x|\sqrt{2t^2-|x|^2}$ we have
  \[
    f(t) = \frac{t}{\sqrt{r_a^2+r_b^2}} = \frac{t}{\sqrt{t^2-|x|\sqrt{2t^2-|x|^2}+r_b^2}}
  \]
  and
  \[
    f'(t) = \frac{r_b^2 \sqrt{2t^2-|x|^2} -t^2 |x|+|x|^3}{\sqrt{2t^2-|x|^2}(t^2+r_b^2-|x|\sqrt{2t^2-|x|^2})^{3/2}}.
  \]
  The denominator of $f'(t)$ is positive, because $t >|x|$, $r_b>0$ and $t^2-|x|\sqrt{2t^2-|x|^2} > 0$ is equivalent to $(t^2-|x|)^2 > 0$. The numerator of $f'(t)$ equals zero, when
  \[
    t = \frac{r_b^4 + r_b^2\sqrt{r^4+|x|^2}+|x|^4}{|x|} > 0
  \]
  and the function $f(t)$ obtains its minimum either as $t \to |x|$ or as $t \to \infty$. We have 
  \[
    \lim_{t \to |x|} f(t) = \frac{|x|}{r_b}
  \]
  and
  \[
    \lim_{t \to \infty} f(t) = \lim_{t \to \infty} \frac{t}{\sqrt{t^2-|x|\sqrt{2t^2-|x|^2}+r_b^2}} = 1.
  \]
  
  Let us now consider $\delta_{G_2}(0,y)$. If $|b-x|<|x|$, then $|a-b|/|b| \le |x-b|/|x|$ and the suprema in the definition of the Seittenranta metric are obtained at $a=x$, $b=\infty$ implying
  \[
    \delta_{G_2}(0,y) = \log \left( 1+\frac{|y|}{|x|} \right).
  \]
  If $|b-x|\ge|x|$, then $|a-b|/|b| \le 1$ and the suprema in the definition of the Seittenranta metric are obtained at $a=\infty$, $b=x$ implying
  \[
    \delta_{G_2}(0,y) = \log \left( 1+\frac{|y|}{|x|} \right).
  \]
  
  (2) It is easy to see that $\partial G_3$ consists of $\partial G_2$ and a line $L$. If both $a,b \in \partial G_2$, then the assertion follows from (1). We assume $b \in L$ and denote $|b+x|=s$. For any $a \in \partial G_3$ with $|a|=t$ we may assume that $|a-b|$ is maximal. This immediately implies that $a \in \partial G_2$. We denote $|x-a|/\sqrt{2} = r$. Now for $r,s \ge 0$ we have by the Pythagorean theorem
  \[
    |a-b|^2 = (s+r)^2 + (2|x|+r)^2, \quad |a|^2 = r^2 + (|x|+r)^2,
  \]
  and we want to find maximum of
  \[
    f(r) = \frac{|a-b|^2}{|a|^2} = \frac{(s+r)^2 + (2|x|+r)^2}{r^2 + (|x|+r)^2}, \quad r \ge 0.
  \]
  We show that $f(r)$ is a decreasing function. Differentiation gives
  \[
    f'(r) = \frac{-2(|x|(2r^2+s^2)+|x|^2 (6r-s)+2 r s (r+s)+2 |x|^3)}{(r^2 + (|x|+r)^2)^2}
  \]
  and the numerator of $f'(r)$ equals zero whenever $r = h(s)$ for
  \[
    h(s) = -s^2-3 |x|^2 \pm \sqrt{s^4 - 2 s^3 |x| + 6s^2 |x|^2 - 2 s |x|^3 + 5 |x|^4}.
  \]
  If we choose minus in $\pm$ then clearly $h(s) < 0$. If we choose plus in $\pm$, then the equation $h(s)=0$ has solution $s=-|x|$ and $s=(|x| \pm \sqrt{7}|x|i)/2$. Now $h(s)$ is either positive or negative for $s \ge 0$. We estimate
  \begin{eqnarray*}
    h(1) & = & -1-3|x|^2+\sqrt{1-2|x|+6|x|^2-2|x|^3+5|x|^4}\\
    & = & -1-3|x|^2+\sqrt{(1+|x|^2)(1-2|x|+5|x|^2)}\\
    & \le & -1-3|x|^2+\sqrt{(1+|x|^2)(1+5|x|^2)}\\
    & \le & -1-3|x|^2+\sqrt{1+6|x|^2+5|x|^4} < 0.
  \end{eqnarray*}
  We have obtained $h(s) < 0$ and since
  \[
    f'(s/6) = \frac{-2(|x|(2(s/6)^2+s^2)+2 (s/6) s ((s/6)+s)+2 |x|^3)}{((s/6)^2 + (|x|+(s/6))^2)^2} < 0
  \]
  we have $f'(r)<0$ for all $r \ge 0$.
  
  Now we are ready to consider $\delta_{G_3}(0,y)$. For any fixed $b \in \partial G_3$, the largest value for $|a-b|/|a|$ is obtained for $a=x$. Thus the suprema in the definition of the Seittenranta metric are obtained at $a=x$, $b=-x$ implying
  \begin{eqnarray*}
    \delta_{G_3}(0,y) & = & \sup_{a,b \in \partial G} \log \left( 1+ \frac{|a-b||y|}{|a||b-y|} \right)\\
    & = & \sup_{b \in \partial G} \log \left( 1+ \frac{|x-b||y|}{|x||b-y|} \right)\\
    & = & \log \left( 1+\frac{2|x||y|}{|x|(|x|+|y|)} \right),
  \end{eqnarray*}
  where the second equality is obtained by the Pythagorean theorem as
  \[
    g(s) = \frac{|x-b|^2}{|b-y|^2} = \frac{4|x|^2+s^2}{(|x|+|y|)^2+s^2} \quad \textrm{and} \quad g'(s) = 2s \frac{(|x|+|y|)^2-4|x|^2}{((|x|+|y|)^2+s^2)^2} < 0
  \]
   implying $g(s) \le g(0)$ and $\sup |x-b|/|b-y| = \sqrt{g(0)} = 2|x|/(|x|+|y|)$.
\end{proof}

Now we are ready to find bounds for $f_\delta(t)$ and $f_\delta'(0)$.

\begin{theorem}\label{thm:seittenranta distance}
  Let $G \subsetneq \Rn$ be a starlike domain. Then
  \[
    f_\delta(t) \ge \log \frac{d_G(0)+t}{d_G(0)-t}
  \]
  for $t \in [0,|z|)$ and
  \[
    f_\delta(t) \le \log \left( 1+\frac{t}{d_G(0)-t}+\frac{t}{d_G(0)} +\frac{t^2}{d_G(0)(d_G(0)-t)} \right)
  \]
  for $t \in [0,d_G(0))$. Moreover, $1 / d_G(0) \le f_\delta'(0) \le 2 / d_G(0)$ and the bounds for $f_\delta'(0)$ are best possible.
\end{theorem}
\begin{proof}
  Lower bound for $f_\delta(t)$ and $f_\delta'(0)$ follow from Proposition \ref{pro:lower bound for delta} and Theorem \ref{thm:f for alpha}.
  
  For the upper bound we use Lemma \ref{lem:upper bound for delta} to obtain
  \begin{eqnarray*}
    \delta_G(0,y) & \le & \log \left( 1+\frac{|y|}{d_G(0)}+\frac{|y|}{d_G(y)}+\frac{|y|^2}{d_G(0)d_G(y)} \right)\\
    & \le & \log \left( 1+\frac{t}{d_G(0)-t}+\frac{t}{d_G(0)} +\frac{t^2}{d_G(0)(d_G(0)-t)} \right)
  \end{eqnarray*}
  for $t \in [0,d_G(0))$. Differentiation and taking the limit gives $2d_G(0)/(d_G(0)^2-t^2) \to 2/d_G(0)$ as $t \to 0$, which proves the upper bound for $f_\delta'(0)$.
  
  Finally, we show that the bounds for $f_\delta'(0)$ are best possible. Sharpness of the lower bound occurs in the domain $G_2$ defined in \eqref{eqn:domain G2}. By Lemma \ref{lem:formula for delta in special domains} (1)
  \[
    \delta_{G_2}(0,y) = \log \left( 1+\frac{t}{d_G(0)} \right)
  \]
  and by differentiation we obtain $1/(d_G(0)+t)$. Now as $t \to 0$ we get the $1/(d_G(0)$.
  
  Sharpness of the upper bound occurs in the domain $G_3$ defined in Lemma \ref{lem:formula for delta in special domains}. Now by Lemma \ref{lem:formula for delta in special domains} (2)
  \[
    \delta_{G_3}(0,y) = \log \left( 1+\frac{2 t d_G(0)}{d_G(0)(d_G(0)+t)} \right).
  \]
  By differentiation and taking the limit we obtain $2 d_G(0)/((d_G(0)+t)(d_G(0)+3t)) \to 2/d_G(0)$ as $t \to 0$.
\end{proof}

\section{Visual angle distance, $v_G$ ($\tau_G$)}

Let $G \subsetneq \Rn$ be a domain. For $u,w \in G$ we define the visual angle distance by
\[
  v_G(u,w) = \sup_{a \in \partial G} \measuredangle(u,a,w).
\]
The visual angle distance was introduced in the 2010's in \cite{KleLinVuoWan2014}. Clearly $v(u,w) \in [0,\pi]$ and therefore we define for $u,w \in G$ a hyperbolic type distance
\[
  \tau_G(u,w) = \tan \frac{v_G(u,w)}{2}.
\]

\begin{proposition}\label{pro:increasing visual angle distance}
  Let $l \subset \Rn \setminus \{ 0 \}$ be a line and $u,w \in l$. Then the function
  \[
    h(\beta) = v_{\Rn \setminus \{ 0 \}}(u,u+\beta (w-u) )
  \]
  is strictly increasing for $\beta \ge 0$.
\end{proposition}
\begin{proof}
  Now $v_{\Rn \setminus \{ 0 \}}(u,a) = \measuredangle (u,0,a)$ and the point $a = a(\beta) = u+\beta (w-u)$ moves along the line $l$. It is easy to see that $a(0) = u$, $a(1) = w$ and $|u-a|$ is strictly increasing as a function of $\beta$.
  
  By the law of sines
  \[
    \frac{|u-a|}{\sin \measuredangle(u,0,a)} = \frac{|u|}{\sin \measuredangle(0,a,u)}
  \]
  and since $\measuredangle(a,u,0)$ is constant we obtain for $c=\pi-\measuredangle(a,u,0) \in (0,\pi)$
  \[
    |u-a| = |u| \frac{\sin \measuredangle(u,0,a)}{\sin \measuredangle(0,a,u)} = |u| \frac{\sin \measuredangle(u,0,a)}{\sin ( c-\measuredangle(u,0,a) )} = |u| \frac{v_{\Rn \setminus \{ 0 \}}(u,a)}{\sin ( c-v_{\Rn \setminus \{ 0 \}}(u,a) )}.
  \]
  Now $|u|$ is a constant and $|u-a|$ is strictly increasing. Thus we need to show that for $c \in (0,\pi)$ the function
  \[
    h(\gamma) = \frac{\gamma}{c-\gamma}
  \]
  is strictly increasing in $(0,c)$. By differentiation we obtain
  \[
    h'(\gamma) = \frac{\sin c}{(\sin(c-\gamma))^2} > 0
  \]
  and the assertion follows.
\end{proof}

In Proposition \ref{pro:increasing visual angle distance} the line $l$ did not contain 0. If $l \subset \Rn$ with $0 \in l$, then for any $u,w \in l \setminus \{ 0 \}$ we have
\begin{equation}\label{eqn:v special case}
  v_{\Rn \setminus \{ 0 \}}(u,w ) = \left\{ \begin{array}{ll}
    0, & \textrm{if } 0 \not\in [u,w],\\
    \pi, & \textrm{if } 0 \in [u,w].\\
  \end{array}   \right.
\end{equation}

\begin{theorem}\label{thm:tau distance}
  Let $G \subsetneq \Rn$ be a starlike domain. Then
  \[
    0 \le f_\tau(t)
  \]
  for all $t \in [0,|z|)$ and
  \[
    f_\tau(t) \le \frac{t}{\sqrt{d_G(0)^2-t^2}+d_G(0)}
  \]
  for all $t \in [0,d_G(0))$. Moreover, the bounds for $f_\tau(t)$ are best possible and $0 \le f_\tau'(0) \le 1/(2 d_G(0))$.
\end{theorem}
\begin{proof}
  By \eqref{eqn:v special case} we note that $0 \le f_\tau(t)$ for all $t \in [0,|z|)$ and by Proposition \ref{pro:increasing visual angle distance} we obtain $0 \le f_\tau'(t)$ for all $t \in [0,|z|)$. The lower bound $0$ is obtained in the domain $G_2$ introduced in \eqref{eqn:domain G2}.
  
  The upper bound for $f_\tau(t)$ is attained in $G'=B^n(0,d(0))$. By \cite[(3.3)]{KleLinVuoWan2014}, $v_\Bn(0,u) = \arcsin |u|$ and thus
  \[
    \tau_{G'}(0,u) = \tan \frac{\arcsin \frac{t}{d_G(0)}}{2} = \frac{t}{\sqrt{d_G(0)^2-t^2}+d_G(0)} \ge \tau_G(0,u) = f_\tau(t).
  \]
  By differentiation and taking the limit we obtain
  \[
    \frac{d_G(0)}{d_G(0) \sqrt{d_G(0)^2-t^2}+d_G(0)^2-t^2} \to \frac{1}{2d_G(0)}
  \]
  as $t \to 0$.  
\end{proof}

\begin{remark}\label{rem:tau remark}
  As in Remark \ref{rem:sigma remark}, Theorem \ref{thm:tau distance} suggests that we could define $\widetilde{\tau}(u,w) = 2 \tau(u,w)$. For this distance function $f_{\widetilde{\tau}}'(0) \le 1/d_G(0)$.
\end{remark}

\section{Discussion}

In this final section we prove our main result and consider the estimates in domains other than starlike.

Until now we have considered starlike domains and found estimates for the function $f_m(t)$. Moreover, the upper bound for $f_m(t)$ is obtained only when $t$ is close to the origin ($t \in [0,d_G(0))$), whereas the lower bound is valid also for large values ($t \in [0,|z|)$).

Let now $G \subsetneq \Rn$ be any domain with $0 \in G$ and $z \in \partial G$. Our results hold also in $G$, when $y \in B^n(0,d_G(0))$. In other words, the lower and upper bounds for $f_m(t)$ are true for $t \in [0,d_G(0))$ and the results for $f_m'(0)$ are also true. We initially choose $0$ to simplify notation. It could be any point in $G$.

\begin{proof}[Proof of Theorem \ref{thm:main theorem}]
  The result for $j_G$ follows from Theorem \ref{thm:distance ratio estimates}, for $k_G$ from Theorem \ref{thm:quasihyperbolic estimates}, for $\widetilde{\sigma}_G$ form Theorem \ref{thm:sigma distance} and Remark \ref{rem:sigma remark}, and for $c_G$ from \ref{thm:Cassinian distance}.
  
  The result for $\alpha_G$ follows form Theorem \ref{thm:f for alpha}, for $\delta_G$ from \ref{thm:seittenranta distance}, and for $\widetilde{\tau}_G$ from Theorem \ref{thm:tau distance} and Remark \ref{rem:tau remark}.
\end{proof}

In Theorem \ref{thm:main theorem} we assumed $z \in \partial G$. This is not needed, if we only consider $f(t)$ close to the origin. We can define for any $z \in \Rn \setminus \{ 0 \}$
\[
  F_m(t) = m(0,t \tfrac{z}{|z|}), \quad t \in [0,d_G(0)).
\]
As an application of Theorem \ref{thm:main theorem} we obtain
\begin{corollary}
  Let $G \subsetneq \Rn$ be a domain with $0 \in G$. For $m \in \{ j,k,\widetilde{\sigma}, c \}$
  \[
    F_m'(0) = \frac{1}{d_G(0)}
  \]
  and for $m \in \{ \alpha, \delta \}$
  \[
    \frac{1}{d_G(0)} \le F_m'(0) \le \frac{2}{d_G(0)},
  \]
  where the upper and lower bounds are best possible, and
  \[
    0 \le F_{\widetilde{\tau}}'(0) \le \frac{1}{d_G(0)}.
  \]
\end{corollary}

If we restrict to starlike John domains, then we can easily get lower bound for $f_m(t)$. A domain $G$ is $C$-John domain, $C \ge 1$, if there is a distinguished point $u_0 \in G$ such that any $u \in G$ can be connected to $u_0$ by a rectifiable curve $\gamma \colon [0,l] \to G$, which is parametrised by arclength and with $\gamma(0)=u$, $\gamma(l)=u_0$ and
\[
  \textrm{dist} (\gamma(t'),\partial G) \ge \frac{1}{c}t'
\]
for every $t' \in [0,l]$. We can choose $u_0 = 0$ and it is easy to see that for example $g(t) \ge \frac{1}{c}t' = \frac{1}{c}(|z|-t)$.

We consider next monotonicity of the function $f_m(t)$. By \cite[Theorem 4.8]{Kle2008} the metric balls $B_j(0,r) = \{ u\in G \colon j_G(0,u)<r \}$ are starlike whenever $G$ is starlike. This implies that in starlike domains $f_j(t)$ is an increasing function. Same is also true for the quasihyperbolic distance \cite[Theorem 2.10]{Kle2008b}, the Apollonian distance \cite[Theorem 3.5]{Kle2013}, the triangular ratio distance \cite[Theorem 1.2]{HarKleVuo15} and the visual angle distance (by definition). For the Seittenranta distance this is not know \cite[Open problem 4.10 (1)]{Kle2013}.

Monotonicity is not true for the function $g(t)$. However, in convex domains $g(t)$ behaves well: there exists a point $a \in [0,|z|)$ such that $g(t)$ is increasing on $[0,a]$ and decreasing on $[a,|z|)$. The next example shows that $g(t)$ has not this property in starlike domains.

\begin{example}\label{exm:comb}
  We show that $g(t)$ misbehaves in the domain
  \[
    G = \Bn \setminus \left( \bigcup_{l=0}^\infty \left[a_l,\tfrac{a_l}{|a_l|}\right]  \right), \quad a_l = (1-2^{-l})e_1 + 2^{-(l+1)} e_2,
  \]
  for $z = 1 \in \partial G$. $G$ is not strictly starlike, but by replacing line segments $[a_l,a_l / |a_l|]$ with angular domains $S_{\alpha_l,a_l,a_l / |a_l|}$, where $\alpha_l$ is small enough, we could construct a strictly starlike domain with the same effect. We stick to the line segment version to make the computation easier to follow. We demonstrate that on each interval $[1-2^{-l},1-2^{-(l+1)}]$ the function $g(t)$ obtains its maximum in $(1-2^{-l},1-2^{-(l+1)})$ and $g(1-2^{-1}) > g(1-2^{-(l+1)})$. This means that $g(t)$ has infinitely many local maxima and minima.
  
  \begin{figure}[ht!]
    \begin{center}
      \includegraphics[width=0.3\textwidth]{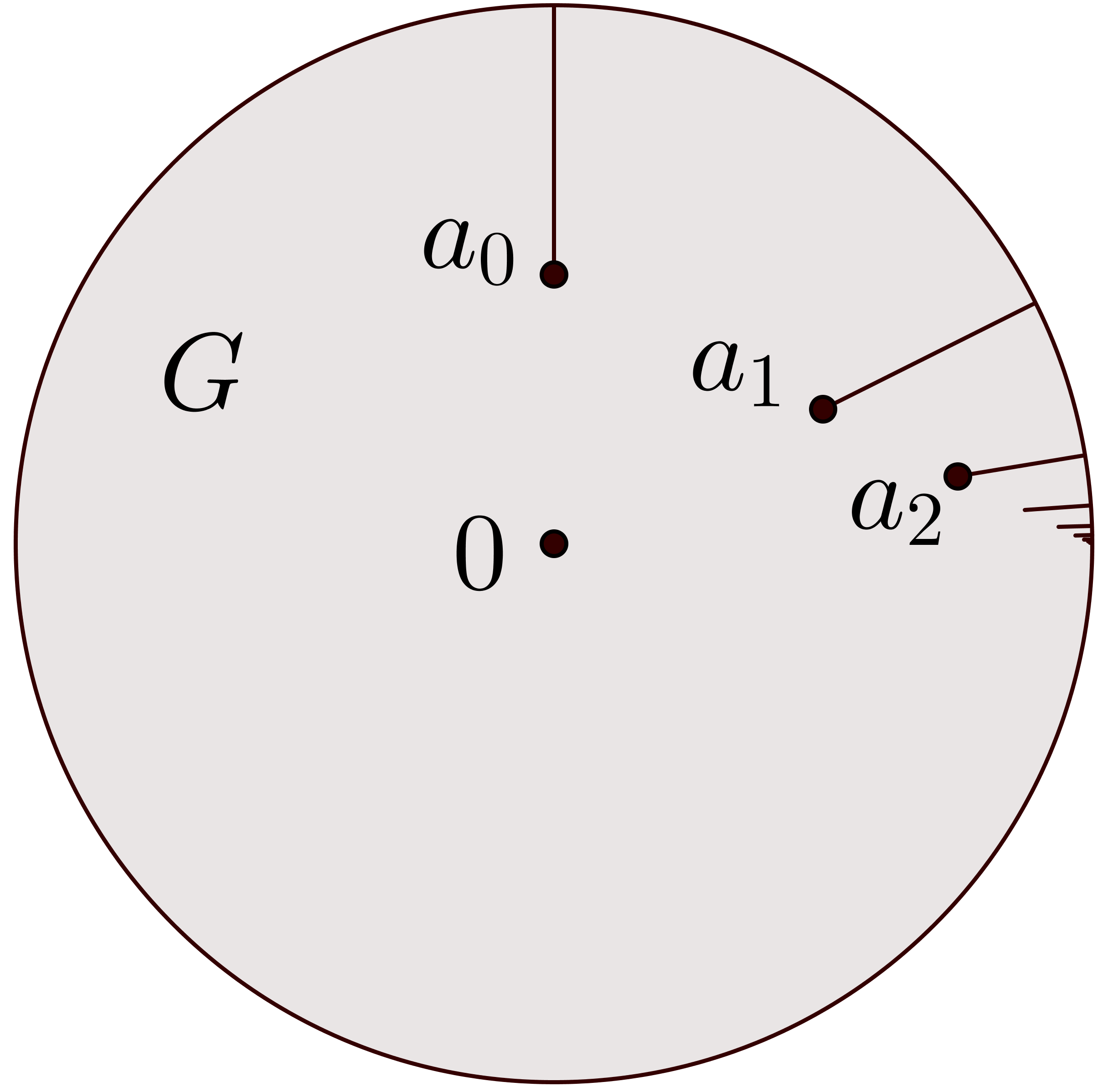}\hspace{5mm}
      \includegraphics[width=0.5\textwidth]{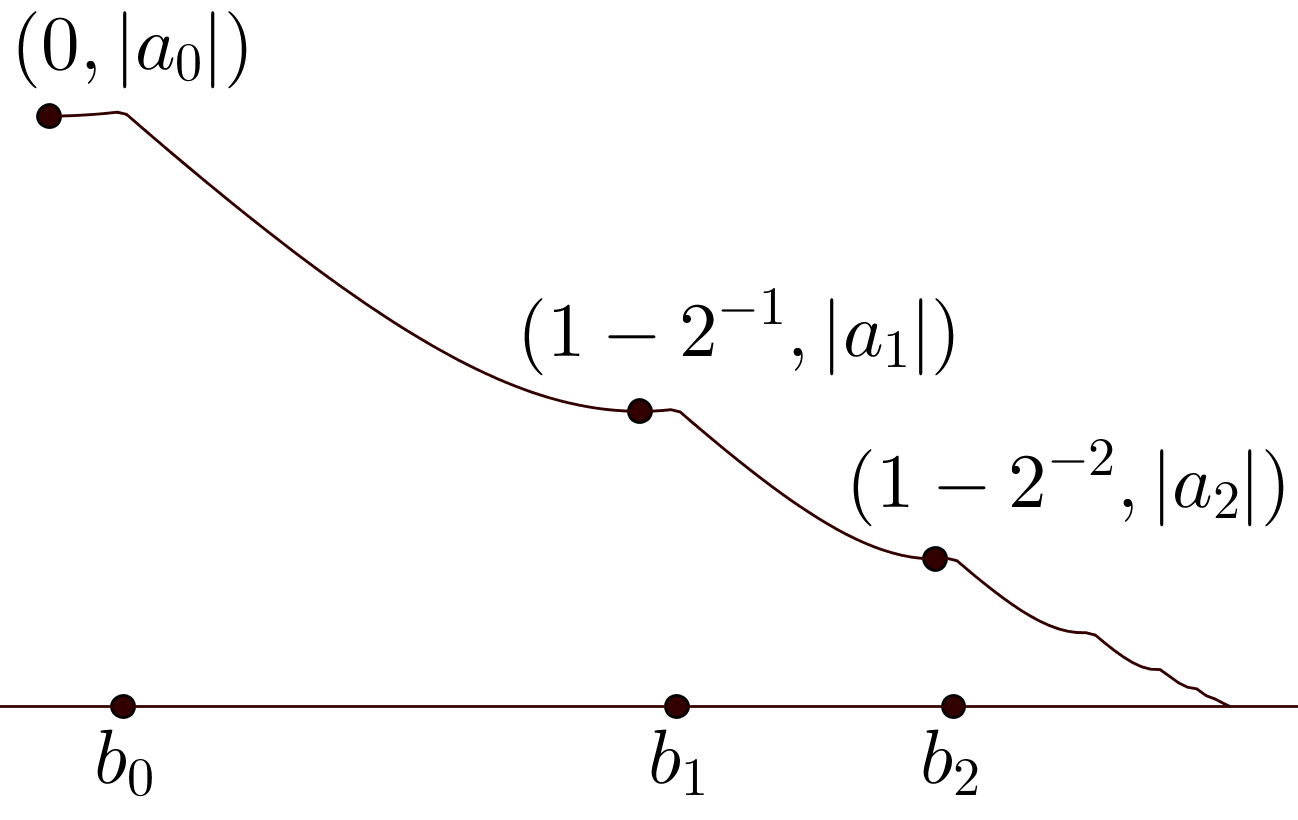}
      \caption{Left: The comb domain $G$ of Example \ref{exm:comb}. Right: Function $g(t)$ in Example \ref{exm:comb}.}
    \end{center}
  \end{figure}
 
  Let us fix $l \in \{ 0,1,2,\dots \}$. Now $g(1-2^{-l}) = 2^{-(l+1)}$ and thus $g(1-2^{-(l+1)}) = 2^{-(l+2)} < g(1-2^{-l})$. We show that for $b_l = 1-\tfrac{7}{8} 2^{-(l+1)}$,
  \begin{equation}\label{eqn:comb inequality}
    g(b_l) > g(1-2^{-l}).
  \end{equation}
  By construction of $G$ it is clear that
  \[
    g(b_l) = \min \{ |b_le_1-a_l|,d(b_le_1,(a_l,a_l/|a_l|]),|b_le_1-a_{l+1}| \}.
  \]
  By geometry we have
  \begin{equation}\label{eqn:comb estimate1}
    d(b_le_1,(a_l,a_l/|a_l|]) \ge 2^{-(l+1)} = g(1-2^{-l}).
  \end{equation}
  We calculate
  \[
    |b_le_1-a_l| = \sqrt{(2^{-(l+1)}/8)^2 + (2^{-(l+1)})^2} = \sqrt{65} \cdot 2^{-(l+4)}
  \]
  and
  \[
    |b_le_1-a_{l+1}| = \sqrt{(7 \cdot 2^{-(l+1)}/8)^2 + (2^{-(l+2)})^2} = \sqrt{65} \cdot 2^{-(l+4)}
  \]
  implying
  \begin{equation}\label{eqn:comb estimate2}
    g(b_l) = \sqrt{65} \cdot 2^{-(l+4)} > \sqrt{64} = 2^{-(l+1)} = g(1-2^{-l}).
  \end{equation}
  Inequalities \eqref{eqn:comb estimate1} and \eqref{eqn:comb estimate2} imply \eqref{eqn:comb inequality} and thus $g(t)$ misbehaves (it has infinitely many local maxima and minima). \hfill $\triangle$
\end{example}

\textbf{Acknowledgements.} The author wishes to thank M. Vuorinen for asking the question about growth of hyperbolic type distances in starlike domains.


\noindent\small{\textsc{R. Kl\'en}}\\
\small{Department of Mathematics and Statistics,
FI-20014 University of Turku, Finland, and Institute of Natural and Mathematical Sciences, Massey University, Auckland, New Zealand}\\
\footnotesize{\texttt{riku.klen@utu.fi}}

\end{document}